\documentclass[10pt]{amsart}
\usepackage{amsthm}
\usepackage{amsmath}
\usepackage{bbm}
\usepackage{nicefrac}
\usepackage{graphicx}
\usepackage{amssymb}
\usepackage{color}
\usepackage{verbatim}
\usepackage{cite}
\usepackage[font=small,skip=5pt]{caption}
\usepackage{enumitem}
\usepackage[makeroom]{cancel}
\usepackage{mathtools}
\usepackage[most]{tcolorbox}
\usepackage[all]{xy}
\usepackage{tikz-cd}
\usepackage{hyperref}
\hypersetup{
	colorlinks=true,
	linkcolor=blue,
	filecolor=magenta,      
	urlcolor=blue,
    citecolor=blue,
}
\usepackage{mathrsfs}
\usepackage{todonotes}

\newtheorem{theorem}{Theorem}[section]
\newtheorem{lemma}[theorem]{Lemma}
\newtheorem{corollary}[theorem]{Corollary}
\newtheorem{proposition}[theorem]{Proposition}
\newtheorem{conjecture}[theorem]{Conjecture}
\newtheorem{problem}[theorem]{Problem}
\newtheorem{remark}[theorem]{Remark}
\newtheorem{definition}[theorem]{Definition}

\newcommand{\sshf}[1]{\mathscr{O}_{#1}}

\newcommand{\prj}[1]{\mathbb{P}^{#1}}

\newcommand{\iso}{\simeq}

\newcommand{\paren}[1]{\left(#1\right)}

\newcommand{\conv}[1]{\textrm{Conv}(#1)}

\newcommand{\height}[1]{\textrm{ht}(#1)}
\newcommand{\ord}[1]{\textrm{ord}(#1)}

\newcommand{\RR}{\mathbb{R}}
\newcommand{\CC}{\mathbb{C}}
\newcommand{\ZZ}{\mathbb{Z}}

\title{The integral closedness of lattice simplices with large lattice length}

\author{
Lei Song, Huanqi Wen and Zhixian Zhu} 

\begin{document}
\maketitle

\begin{abstract}
We prove that every $n$-dimensional lattice simplex $P$ whose lattice length $L(P)\ge n-1$ is integrally closed. As an application, we obtain a simple criterion for the projective normality of ample line bundles on  $\mathbb{Q}$-factorial toric Fano varieties with Picard number one. We further obtain a refinement of this result in terms of the invariant $\Gamma_{P}$ introduced in \cite{GonZhu2022}.
\end{abstract}


\section{Introduction}
Let $M$ be a free $\mathbb{Z}$-module of rank $n\ge 1$ and let $M_{\RR}=M\otimes_{\ZZ}\RR$ be the associated $n$-dimensional vector space. A polytope $P\subseteq M_{\RR}$ is defined as the convex hull of a finite set $\{u_0,\dots,u_m\}\subseteq M_{\RR}$, denoted by $P=\conv{u_0,\cdots,u_m}$. After removing redundant elements, we may assume that $u_0,\dots,u_m$ are the vertices of $P$. We call $P$ a lattice polytope if $u_i\in M$ for all $i$. Throughout this paper, all polytopes are assumed to be full-dimensional. 

Let $P,\ P_1,\ P_2$ be polytopes in $M_{\mathbb{R}}$ and $r\in \mathbb{R}$. The Minkowski sum $P_1+P_2:=\{u_1+u_2\: | \: u_1\in P_1, u_2\in P_2\}$ and the dilation by scalar $rP:=\{ru \: | \: u\in P\}$ . When $r$ is a natural number, $rP=\sum^r_{i=1} P$.

A fundamental problem concerning lattice polytopes is: 
\begin{problem}
    For which lattice polytope $P$ does the equality
    \begin{equation}
    \label{eq: integrally closed P}
        P\cap M+(rP)\cap M=(r+1)P\cap M
    \end{equation}
    hold for all positive integers $r$?
\end{problem}

A lattice polytope that satisfies the equality \eqref{eq: integrally closed P}
for all $r\in \ZZ_{>0}$ is called \textit{integrally closed}. From the perspective of algebraic geometry, the integral closedness of $P$ is equivalent to the projective normality of the polarized toric variety $(X,L)$ associated with $P$. 

Let us briefly recall some basics in toric geometry (cf. \cite{FultonIntroToToric,CoxToric}). Let $N=\text{Hom}(M,\ZZ)$. Given a lattice polytope $P\subseteq M_{\RR}$, one associates to $P$ a normal fan $\Delta_P$, which defines a projective toric variety $X$ together with an ample Cartier divisor $D=D_P$. The associated ample line bundle is $L=\mathcal{O}_X(D)$. Denote by $T$ the algebraic torus in $X$. For $u\in M$, let $\chi^u$ denote the corresponding regular function on $T$, also viewed as a rational function on $X$. Then one can describe the global sections of $L$ through the lattice points in $P$: 
\begin{equation}
    \label{eq: global section of L}
    H^0(X,L)\cong \bigoplus_{u\in P\cap M} \CC\cdot \chi^u.
\end{equation}
Moreover, the dilated polytope $rP$ corresponds to the r-fold tensor product $L^{\otimes r}$, and the multiplication map
\begin{equation}
\label{eq: multiplication map}
    H^0(X,L)\otimes H^0(X,L^{\otimes r})\to H^0(X,L^{\otimes (r+1)})
\end{equation}
sends $\chi^{u_1}\otimes \chi^{u_2}$, where $u_1\in P\cap M$ and $u_2\in rP\cap M$, to $\chi^{u_1+u_2}$ under the identification \eqref{eq: global section of L}. Therefore, the equality \eqref{eq: integrally closed P} amounts to the surjectivity of the multiplication
map \eqref{eq: multiplication map}. Recall that a basepoint free line bundle $L$ is said to be projectively normal if the multiplication map \eqref{eq: multiplication map} is surjective for all $r\in \ZZ_{>0}$. Hence, $L$ is projectively normal if and only if $P$ is integrally closed.

Projective normality for a general smooth projective variety is already highly nontrivial even for surfaces. For instance, Mukai's conjecture (cf. \cite{EL93}) remains widely open at the level of projective normality. Inspired by the Mukai conjecture, the first and third authors proposed the following conjecture in the toric setting, see \cite{SWZ24}.
\begin{conjecture}\label{Conj: Song-Zhu}
    Let $P\subseteq M_{\RR}$ be an $n$-dimensional lattice polytope. If $L(P)\ge n-1$, then the equality \eqref{eq: integrally closed P} holds for all $r\in \ZZ_{>0}$, i.e., $P$ is integrally closed.
\end{conjecture}

Here $L(P)$ denotes the lattice length of $P$ (see Definition \ref{def: lattice length}). In the toric setting, we can interpret the lattice length in terms of intersection numbers. More precisely, each edge $e$ of $P$ corresponds to a $T$-invariant curve $C$ on $X$, and the intersection number $L\cdot C$ is equal to the lattice length of $e$, where $L$ is the line bundle associated with $P$ as above. By Definition \ref{def: lattice length}, the lattice length of $P$ is the minimum of the lattice lengths of all edges of $P$. Therefore, Conjecture \ref{Conj: Song-Zhu} can be rephrased as follows: if $L\cdot C\ge n-1$ for every $T$-invariant curve $C$, then $L$ is projectively normal.

Several partial results toward this conjecture are known. In \cite{OgataNakagawa2002}, Ogata and Nakagawa proved that every dilated polytope $P=sQ$ with $s\ge n-1$ is integrally closed. 
By Gubeladze \cite{Gubeladze2012}, $P$ is projectively normal provided that $L(P)\ge 4n(n+1)$, and the condition can
be relaxed to $L(P)\ge n(n+1)$ for lattice simplices. However, it is desirable to have a condition that is linear in dimension. In \cite{SWZ24}, by special covering, we proved the conjecture in dimension $3$ for every lattice simplex and in dimension $4$ for most lattice simplices.

In this paper, we prove Conjecture \ref{Conj: Song-Zhu} for lattice simplices.
\begin{theorem}
\label{thm: main thm 1}
    Let $P\subseteq M_{\mathbb{R}}\iso \mathbb{R}^n$ be a lattice simplex. Suppose $L(P)\ge n-1$. Then $P$ is integrally closed.  
\end{theorem}

As explained in \cite{Fujino03}, every $\mathbb{Q}$-factorial toric Fano variety with Picard number one can be obtained from a lattice simplex, so we have
\begin{corollary}
Let $X$ be an $n$-dimensional $\mathbb{Q}$-factorial toric Fano variety with Picard number one and $L$ be an ample line bundle on $X$. Suppose $L\cdot C\ge n-1$ for any $T$-invariant curve $C$ on $X$. Then $L$ is projectively normal. \qed
\end{corollary}

In \cite{GonZhu2022}, González and the third author introduced an invariant $\Gamma_P$ associated with $P$ to study the k-jet ampleness of line bundles on toric varieties. The invariant $\Gamma_P$ is a nonnegative rational number that measures the singularity of $P$: the more singular $P$ is, the larger $\Gamma_P$ becomes. In general, $\Gamma_P\le n-2$ whenever $n\ge 2$.  

The same line of reasoning as in the proof of Theorem \ref{thm: main thm 1} yields a more refined statement when $P$ is ``sufficiently singular".
\begin{theorem}\label{thm: main thm 2}
Let $P$ be an $n$-dimensional lattice simplex in $M_{\RR}$. Assume that either
\begin{itemize}
    \item $n=2$, or
    \item $n\ge 3$ and $\Gamma_P\ge (\frac{2}{3})n-1$.
\end{itemize}
If $L(P)\ge \Gamma_P+1$, then $P$ is integrally closed.
\end{theorem}

\textbf{Acknowledgments}.
During the preparation of the article, L.S. was partially supported by NSFC grants No.~12471043, No.~12371063 and  Guangdong Basic and Applied Basic Research Foundation No.~ 2025A1515012258, and Z.Z. was partially supported by NSFC grant No.~12101423. 
\vspace{2mm}

\section{Preliminaries}
In this section, we introduce some definitions and fix notation.

Let $P=\conv{u_0,\dots,u_m}$ be a lattice polytope in $M_{\RR}$, with vertices $u_0,u_1,\ldots,u_m$. 

\begin{definition}\label{def: lattice length}
    For any pair of distinct vertices $u_i, u_j$ of $P$, let $e_{ij}=\{\lambda u_i+(1-\lambda)u_j\ |\ 0\le \lambda\le 1\}$ denote the line segment joining $u_i$ and $u_j$. The  \emph{lattice length} of $e_{ij}$ is defined by \[l_{ij}:=\#(e_{ij}\cap M)-1.\] We refer to $l_{ij}$ as the lattice length of $e_{ij}$, regardless of whether $e_{ij}$ is an edge of $P$. The \emph{lattice length} of $P$ is defined by 
    \[
        L(P):=\min \{l_{ij}\mid e_{ij}\text{ is an edge of } P\}.
    \]
\end{definition}

For $i\ne j$, set 
\[u_{ij}:=u_j-u_i \text{ and }\hat{u}_{ij}:=\frac{u_{ij}}{l_{ij}}.\]
Thus $\hat{u}_{ij}$ is the primitive lattice vector in the direction from $u_i$ to $u_j$.

For each vertex $u_i$ of $P$, let $u_{j_1},\cdots,u_{j_k}$ be the vertices adjacent to $u_i$. We define $C(P,u_i)$ to be the cone generated by $u_{ij_1},\cdots, u_{ij_k}$. Equivalently, $C(P, u_i)$ is the tangent cone of $P$ at $u_i$, translated to the origin. Its primitive ray generators are $\hat{u}_{ij_1},\cdots,\hat{u}_{ij_k}$.

Next we define the nonnegative invariant $\Gamma_P$ associated to the lattice polytope $P$. 

\begin{definition}
Let $Q$ be a strictly convex rational polyhedral cone in $M_{\RR}$, and let $w_1,\cdots,w_s$ be the primitive ray generators. Let $\beta_1,\cdots,\beta_t$ be the Hilbert basis of the affine semigroup $Q\cap M$. Every $u\in Q\cap M$ admits an expression in terms of either the ray generators $w_i$ or the Hilbert basis elements $\beta_j$. Define
\[\height{u}=\max\{\sum_{i=1}^s a_i\ |\ u=\sum_{i=1}^s a_i w_i \text{ and } a_i\in \mathbb{R}_{\ge 0} \text{ for all } i\}\]
and
    \[\ord{u}=\max\{\sum_{j=1}^t \lambda_j\mid u=\sum_{j=1}^t \lambda_j \beta_j\text{ and } \lambda_j\in \mathbb{Z}_{\ge 0} \text{ for all } j\}.\]
    Let $S_Q=\{\sum_{i=1}^s a_i w_i\ |\ 0\le a_i<1 \text{ for all } i\}$. We define
    \[
    \Gamma_Q:=\max\{\height{u}-\ord{u}\mid u\in S_Q\cap M\}.
    \]
\end{definition}

\begin{definition}
Given a lattice polytope $P$ with vertices $u_0, \ldots, u_m$, define
     \[\Gamma_P:=\max\{\Gamma_{C(P,u_i)}\mid 0\le i\le m\}.\]
\end{definition}
We will use the following bound.
\begin{lemma}[{\cite[Lemma 2.6]{GonZhu2022}}]\label{Gamma}
    For any lattice polytope $P$ of dimension $n\ge 2$, we have $0\le \Gamma_P\le n-2$. Furthermore, if $P$ is smooth, then $\Gamma_P=0$.
\end{lemma}

\section{A Partition Lemma}
In this section, we present a useful combinatorial partition lemma that provides an effective method of decomposing points. In the following, for a nonnegative integer $l$, let $[l]=\{1, 2, \cdots, l\}$. 
\begin{lemma}\label{harness}
Let $1<r$ be a real number and let $x_0,x_1,\dots,x_l$ be nonnegative real numbers. Let $a$ and $a'$ be real numbers  such that $a\le  r-1, a'\le 1$, and  $$a+a'+\sum\limits^l_{i=0}x_i=r.$$ Assume that $x_i\le x_0$ for each $1\le i\le l$. Then there exists a subset $I$ of $[l]$ satisfying the property
\[a+\sum_{i\in I}x_i\le r-1 \text{ and } a'+\sum_{i\in [l]\backslash I}x_i\le 1.\]
\end{lemma}
\begin{remark}
In practice below, $r$ is an integer, and $a, a'\ge 0$. 
\end{remark}
\begin{proof}
First we observe that either $a+x_1\le r-1$ or $a'+x_1\le 1$. Indeed, if neither inequality holds, we have
\[a+a'+2x_1>(r-1)+1=r,\]
which implies
\[x_1>r-(a+a'+x_1)=\sum^l_{i=2} x_i+x_0\ge x_0,\]
which is a contradiction to the assumption. Thus if $a+x_1\le r-1$, we can take $I_1=\{1\}$; if $a+x_1>r-1$, then $a'+x_1\le 1$, and we can take $I_1=\emptyset$. 
 
Suppose for some $1\le k\le l-1$, we have found a subset $I_k\subseteq [k]$ satisfying
\[a+\sum_{i\in I_k} x_i\le r-1, \text{ and } a'+\sum_{i\in [k]\backslash I_k} x_i\le 1.\] 
We claim that either
\[a+\paren{\sum_{i\in I_k} x_i}+x_{k+1}\le r-1 \text{ or }  a'+\paren{\sum_{i\in [k]\backslash I_k}x_i}+x_{k+1}\le 1,\]
because otherwise, it follows that
\[a+a'+\paren{\sum^k_{i=1} x_i}+2x_{k+1}>(r-1)+1=r,\]
implying
\[x_{k+1}>r-(a+a'+\sum^{k+1}_{i=1}x_i)=\paren{\sum^l_{i=k+2} x_i}+x_0\ge x_0,\]
which is a contradiction to the assumption.

Therefore if $a+\paren{\sum_{i\in I_k} x_i}+x_{k+1}\le r-1$, we can take $I_{k+1}=I_k\cup \{k+1\}$; and if $a+\paren{\sum_{i\in I_k} x_i}+x_{k+1}> r-1$, then it occurs that $a'+\paren{\sum_{i\in [k]\backslash I_k} x_i}+x_{k+1}\le 1$ and we can take $I_{k+1}=I_k$. 

So after $l$ steps, we reach $I=I_l\subseteq [l]$ satisfying the desired property.
\end{proof}

\section{Proof of the Main Theorems}
In this section, we shall prove Theorems \ref{thm: main thm 1} and \ref{thm: main thm 2}. Theorem \ref{thm: main thm 1} follows from Propositions \ref{r>=3} and \ref{decomp for r=2}. The case 
$r\ge 3$ is relatively straightforward, and we obtain a result stronger than what is required for integral closedness, while the case $r=2$ involves Hilbert bases of affine semigroups and will be treated separately. 

The core idea is that every lattice point in 
$rP$ admits a unique representation as a convex combination of vertices when $P$ is a simplex. Lemma \ref{harness} will be applied repeatedly to partition the combination coefficients into two suitable subsets, thereby inducing a decomposition of the sum over lattice points. 

\begin{proposition}\label{r>=3}
Let $P$ be a lattice simplex in $M_{\mathbb{R}}$ with $L(P)\ge n-1$. We have $rP=(r-1)P+P\cap M$ for all $r\ge 3$.
\end{proposition}

\begin{remark}
Note that $rP=(r-1)P+P\cap M$ implies that $rP\cap M=(r-1)P\cap M+P\cap M$.
It is also important to note that Proposition \ref{r>=3} is false for $r=2$. Consider the following example. Let $P$ be the 3-simplex with $u_0=(0, 0, 0), u_1=(2, 0, 0), u_2=(0, 2, 0)$ and $u_3=(0, 0, 2)$. Then 
\begin{eqnarray*}
P\cap M&=&\{(0, 0, 0), (2, 0, 0), (0, 2, 0), (0, 0, 2), (1, 1, 0), \\
&&(1, 0, 1), (0, 1, 1), (1, 0, 0), (0, 1, 0), (0, 0, 1)\}.
\end{eqnarray*}
Let $x=(\frac{14}{15}, \frac{14}{15}, \frac{14}{15})\in 2P$. But for any $y\in P\cap M$, $x-y\notin P$. 
\end{remark}

Proposition \ref{r>=3} is a direct consequence of the following result by taking $m=n$.

\begin{proposition}
\label{r>=3 generalization}
Let $P$ be a lattice polytope of dimension $n$, with $m+1$ vertices. 
Suppose $l_{ij}\ge n-1$ for any $i\neq j$. Then  we have 
\[rP=(r-1)P+P\cap M\]
for $r\ge \lceil\frac{m}{n-1}\rceil+1$.
\end{proposition}

\begin{proof}
    Given $u\in rP$, put $u=\sum^m_{i=0}x_iu_i$ with $x_i\ge 0$ and $\sum^m_{i=0} x_i=r$. We can assume that $x_0$ is the largest among $\{x_i\}$.
    
For each $1\le i\le m$, there exists a unique non-negative integer $t_i$ such that
    \[\frac{t_i}{l_{0i}}\le x_i, \text{ while  }\frac{t_i+1}{l_{0i}}> x_i.\]
    Put $\tilde{x}_i=x_i-\frac{t_i}{l_{0i}}$. By assumption,
    \begin{equation}
    \label{eq: key formula 1}
          \sum_{i=1}^m \tilde{x}_i<\sum_{i=1}^m \frac{1}{l_{0i}}\le \frac{m}{n-1}\le r-1 .
    \end{equation}
    Applying Lemma \ref{harness} (by taking $a=\sum_{i=1}^m \tilde{x}_i$ and $a'=0$), we obtain a set $I\subseteq [m]$ satisfying the property that
\[\sum^m_{i=1}\tilde{x_i}+\sum_{i\in I}\frac{t_i}{l_{0i}}\le r-1, \text{ and } \sum_{i\in [m]\backslash I}\frac{t_i}{l_{0i}}\le 1.\]
Consequently, we can write $u=v_1+v_2$, where
\begin{eqnarray*}
v_1 &=& (r-1)u_0+\sum^m_{i=1}\tilde{x}_i(l_{0i}\hat{u}_{0i})+\sum_{i\in I}\frac{t_i}{l_{0i}}(l_{0i}\hat{u}_{0i})\in (r-1)P,\\
v_2 &=& u_0+\sum_{i\in [m]\backslash I}\frac{t_i}{l_{0i}}(l_{0i}\hat{u}_{0i})\in P\cap M.
\end{eqnarray*}
This finishes the proof.
\end{proof}

\begin{proposition}\label{decomp for r=2}
Let $P$ be a lattice simplex with $L(P)\ge n-1$. We have $$2P\cap M=P\cap M+P\cap M.$$
\end{proposition}

\begin{lemma}\label{small norm}
Let $P$ be a lattice polytope. Suppose $L(P)\ge \Gamma_P+1$. Then at any vertex $u_i$, the elements of the minimal Hilbert basis for $C(P, u_i)$ lie in $(P-{u_i})\cap M$. 
\end{lemma}
\begin{proof}
We can assume $i=0$. Let $y$ be an element of the minimal Hilbert basis of $C(P,u_0)$. Let $u_1,\cdots, u_m$ be the vertices of $P$ adjacent to $u_0$. We write $y=\sum^m_{j=1} a_j\hat{u}_{0j}$ for some $a_j\in \mathbb{Q}\cap [0, 1)$. 

By the definitions of $\text{ht}(y),\ \Gamma_{C(P,u_0)}$, and $\Gamma_P$, and using $\text{ord}(y)=1$, we have
\[\sum^m_{j=1} a_j\le\text{ht}(y)\le \Gamma_{C(P,u_0)}+1\le \Gamma_P+1.\]
As a result,
\[\sum^m_{j=1}\frac{a_j}{l_{0j}}\le \paren{\frac{1}{\Gamma_P+1}}\sum^m_{j=1} a_j\le 1.\]
Finally, noting $y=\sum^m_{j=1}\frac{a_j}{l_{0j}}(l_{0j}\hat{u}_{0j})$, we have $y+u_0\in P\cap M$. 
\end{proof}

\begin{proof}[Proof of Proposition \ref{decomp for r=2}]
To begin  with, we introduce the following notation. For $v \in C(P,u_0)$, write $v=\sum\limits_{i=1}^n x_i u_{0i}$, and define
\[
|v|=\sum\limits_{i=1}^n x_i.
\]

Step 0 (setup): Given $u=\sum^n_{i=0}x_iu_i\in 2P\cap M$, we can assume that $x_0=\max\limits_{0\le i\le n}\{x_i\}$, in particular $x_0\ge \frac{2}{n+1}$. We write
\[u=2u_0+\sum^n_{i=1} x_i(l_{0i}\hat{u}_{0i}).\]

Step 1 (reduction): For each $1\le i\le n$, there exists a unique non-negative integer $t_i$ such that
\[\frac{t_i}{l_{0i}}\le x_i<\frac{t_i+1}{l_{0i}}.\]
Put $\tilde{x}_i=x_i-\frac{t_i}{l_{0i}}$. Then
\[u=2u_0+\sum^n_{i=1} \tilde{x}_i(l_{0i}\hat{u}_{0i})+\sum^n_{i=1} 
\frac{t_i}{l_{0i}}(l_{0i}\hat{u}_{0i}).\]

Note for all $i>0$, $\frac{t_i}{l_{0i}}\le x_i\le x_0$, so by Lemma \ref{harness}, if we are able to write $\tilde{x}:=\sum\limits^n_{i=1} \tilde{x}_i(l_{0i}\hat{u}_{0i})=\tilde{v}+\tilde{v}'$ such that $|\tilde{v}|, |\tilde{v}'|\le 1$ and $\tilde{v}, \tilde{v}'\in (P-u_0)\cap M$, then there exists a subset $I\subseteq [n]$ such that
\[|\tilde{v}|+\sum_{i\in I}\frac{t_i}{l_{0i}}\le 1, \text{ and }
|\tilde{v}'|+\sum_{i\in [n]\backslash I}\frac{t_i}{l_{0i}}\le 1.\]
As a result, we can write $u=v_1+v_2$, where
\begin{eqnarray*}
v_1 &=& u_0+\tilde{v}+\sum_{i\in I}\frac{t_i}{l_{0i}}(l_{0i}\hat{u}_{0i})\in P\cap M,\\
v_2 &=& u_0+\tilde{v}'+\sum_{i\in [n]\backslash I}\frac{t_i}{l_{0i}}(l_{0i}\hat{u}_{0i})\in P\cap M,
\end{eqnarray*}
as asserted.

Step 2: In view of the above step, we can replace $x_i$ by $\tilde{x}_i$. So it remains to consider the case where $x_i<\frac{1}{l_{0i}}$ for all $i>0$. 

Take $y_1, \cdots, y_m$ (not necessarily distinct) in the minimal Hilbert basis for $C(P, u_0)\cap M$ such that
\[u-2u_0=\sum^m_{i=1} y_i.\]
We have
\begin{equation}\label{upper bound for modulus}
\sum^m_{i=1} |y_i|=\sum^n_{i=1}x_i<\sum^n_{i=1}\frac{1}{l_{0i}}\le \frac{n}{n-1}.
\end{equation}
Therefore
\begin{equation}
    \label{eq: key formulu 2}
    2-\sum^m_{i=1}|y_i|>2-\frac{n}{n-1}=\frac{n-2}{n-1}.
\end{equation}

We shall proceed by cases.

Case (i): Assume that $|y_j|\le 2-\sum\limits^m_{i=1}|y_i|$ for all $1\le j\le m$. \\
Then by Lemma \ref{harness}, there exists a subset $J\subseteq [m]$ such that 
$\sum\limits_{j\in J} |y_j|\le 1$ and $\sum\limits_{j\in [m]\backslash J} |y_j|\le 1$. Therefore, $u=v_1+v_2$, where
\begin{eqnarray*}
v_1&=& u_0+\sum_{j\in J} y_j\in P\cap M, \\
v_2&=& u_0+\sum_{j\in [m]\backslash J} y_j\in P\cap M.
\end{eqnarray*}
We are done. 

Case (ii): For some $s$, it holds that $|y_s|>2-\sum\limits_{i=1}^m |y_i|$. Without loss of generality, we may assume that $s = m$. \\
Then we claim
\begin{equation}\label{equality}
\sum^{m-1}_{i=1}|y_i|\le 1.
\end{equation}

For otherwise, 
\[\sum^m_{i=1}|y_i|>1+|y_m|>1+(2-\sum_{i=1}^m |y_i|)>1+\frac{n-2}{n-1},\]
where the last inequality follows from \eqref{eq: key formulu 2}. This immediately contradicts (\ref{upper bound for modulus}) provided $n\ge 3$. When $n=2$, we have 
\[|y_m|> 2-\sum^m_{i=1}|y_i|=x_0\ge \frac{2}{3},\]
which implies that
\[\sum^{m-1}_{i=1}|y_i|\le 2-\frac{2}{3}-|y_m|< 2-2\times \frac{2}{3}<1.\]
Thus in any event, we have established inequality (\ref{equality}). 

Finally, combining (\ref{equality}) together with $|y_m|\le 1$, which is guaranteed by Lemmas \ref{Gamma} and \ref{small norm}, we can write
$u=v_1+v_2$, where
\begin{eqnarray*}
v_1&=& u_0+y_m\in P\cap M,\\
v_2 &=& u_0+\sum^{m-1}_{i=1} y_i\in P\cap M.
\end{eqnarray*}
This completes the proof.
\end{proof}

By adapting the argument in the proof of Theorem \ref{thm: main thm 1}, we obtain the following more general sufficient condition for the integral closedness of a lattice simplex $P$ in terms of the invariant $\Gamma_P$ associated with $P$.

Now we prove Theorem \ref{thm: main thm 2}.
\begin{proof}
    When $n=2$, Lemma \ref{Gamma} implies that $\Gamma_P = 0$. Hence the statement follows from Theorem \ref{thm: main thm 1}.
    
    When $n\ge3$, the proof is analogous to those of Propositions \ref{r>=3 generalization} and \ref{decomp for r=2}. In the setting of Proposition \ref{r>=3 generalization}, we take $m=n$ and rewrite \eqref{eq: key formula 1} as
    \begin{equation}
        \sum_{i=1}^n \tilde{x}_i <\sum_{i=1}^n \frac{1}{l_{0i}}\le \frac{n}{\Gamma_P+1}\le \frac{3}{2},
    \end{equation}
    where the last two inequalities follow from our assumptions. Consequently, if $r-1\ge \frac{3}{2}$, and hence $r\ge 3$, then 
    \begin{equation}
    \label{eq: r>=3}
        rP=(r-1)P+P\cap M.
    \end{equation}

   It remains to show that $P\cap M+P\cap M=2P \cap M$. Following the proof of Proposition \ref{decomp for r=2}, it suffices to reconsider Step 2. Under the present assumption, we have
    \begin{equation}
    \label{eq: upper bound}
        \sum^m_{i=1} |y_i|=\sum^n_{i=1}x_i<\sum^n_{i=1}\frac{1}{l_{0i}}\le \frac{n}{\Gamma_P+1},
    \end{equation}
    and hence,
    \begin{equation}
        \label{eq: lower bound}
        2-\sum_{i=1}^m |y_i|> 2-\frac{n}{\Gamma_P+1}.
    \end{equation}
    These inequalities are analogous to \eqref{upper bound for modulus} and \eqref{eq: key formulu 2}.
    We now proceed by cases. The argument for Case (i), i.e., $|y_j|\le 2-\sum\limits^m_{i=1}|y_i|$ for all $1\le j\le m$, is identical to that in Proposition \ref{decomp for r=2}. For Case (ii), i.e., $|y_m|>2-\sum\limits_{i=1}^m|y_i|$, we still claim that
    \begin{equation}
        \sum_{i=1}^{m-1} |y_i|\le 1.
    \end{equation}
    Otherwise,
    \[\sum_{i=1}^m |y_i|>1+|y_m|>1+(2-\sum_{i=1}^m |y_i|)\ge 3-\frac{n}{\Gamma_P+1},\]
    where the last inequality follows from \eqref{eq: lower bound}.
    This contradicts \eqref{eq: upper bound} provided that $\Gamma_P\ge \frac{2}{3} n-1$. Therefore, we deduce  that $u=(u_0+y_m)+(u_0+\sum\limits_{i=1}^{m-1}y_i)$. Hence we have proved 
    \begin{equation}
    \label{eq: r=2}
         P\cap M+P\cap M=2P\cap M.
    \end{equation}
This completes the proof.
\end{proof}

\section{Some Open Problems}
In\cite{Oda08}, Oda asked whether any smooth lattice polytope $P$ is integrally closed. Since smooth polytopes satisfy $\Gamma_P=0$, it is natural to consider the following general question:

\begin{problem}
Let $P\subseteq M_{\RR}$ be a lattice polytope. If $L(P)\ge \Gamma_P+1$, then is $P$ integrally closed?
\end{problem}
The problem can be viewed as a unifying framework that includes both Oda’s conjecture and Conjecture \ref{Conj: Song-Zhu} as special cases: indeed, in light of Lemma \ref{Gamma}, an affirmative answer would immediately imply the validity of both conjectures. Theorem \ref{thm: main thm 2} provides some evidence in this direction.
 
Finally, in the toric setting, integral  closedness is equivalent to projective normality, which amounts to Property $N_p$  (cf. \cite{Lazarsfeld04}) for $p=0$. Inspired by Mukai's conjecture, the first and third authors propose the following conjecture.
\begin{conjecture}
Let $L$ be an ample line bundle on a projective toric variety $X$ of dimension $n\ge 2$. Suppose $L\cdot C\ge n-1+p$ for any $T$-invariant curve $C$. Then $L$ satisfies Property $N_p$.
\end{conjecture}
The conjecture is true for $n=2$ by \cite{Sch04}, and for arbitrary $n$ but $L=A^{\otimes k}$ with $k\ge n-1+p$ by \cite{HSS06}.
In a recent paper \cite{SW26}, the conjecture is proved for a class of smooth projective toric varieties, including the blowup of $\prj{n}$ at $k$ points in general position, for $0\le k\le n+1$; the blowup of $\prj{n}$ along $\prj{k}$, for $0\le k\le n-1$; and projective bundles $\mathbb{P}(\sshf{\prj{n}}^{\oplus r_1}\oplus \sshf{\prj{n}}(1)^{\oplus r_2})$, for $r_1, r_2\ge 1$. Moreover, the class is closed under taking finite products. 

\bibliographystyle{alpha}
\bibliography{toric}

@article{HSS06,
  title={Syzygies, multigraded regularity and toric varieties},
  author={Hering, Milena and Schenck, Hal and Smith, Gregory G.},
  journal={Compositio Mathematica},
  volume={142},
  number={6},
  pages={1499--1506},
  year={2006},
  publisher={London Mathematical Society}
}

@article {Sch04,
    AUTHOR = {Schenck, Hal},
     TITLE = {Lattice polygons and {G}reen's theorem},
   JOURNAL = {Proc. Amer. Math. Soc.},
  FJOURNAL = {Proceedings of the American Mathematical Society},
    VOLUME = {132},
      YEAR = {2004},
    NUMBER = {12},
     PAGES = {3509--3512},
      ISSN = {0002-9939},
   MRCLASS = {52B35 (14M25)},
  MRNUMBER = {2084071},
       DOI = {10.1090/S0002-9939-04-07523-9},
       URL = {https://doi.org/10.1090/S0002-9939-04-07523-9},
}

@article{GonZhu2022,
title = {Generation of jets and {F}ujita's jet ampleness conjecture on toric varieties},
journal = {Journal of Pure and Applied Algebra},
volume = {226},
number = {4},
pages = {106873},
year = {2022},
issn = {0022-4049},
doi = {https://doi.org/10.1016/j.jpaa.2021.106873},
url = {https://www.sciencedirect.com/science/article/pii/S0022404921002140},
author = {José Luis González and Zhixian Zhu}
}

@article{Fujino03,
AUTHOR = {Fujino, Osamu},
TITLE = {Notes on toric varieties from {M}ori theoretic viewpoint},
JOURNAL = {Tohoku Math. J. (2)},
VOLUME = {55},
YEAR = {2003},
NUMBER = {4},
PAGES = {551--564},
}

@book {FultonIntroToToric,
    AUTHOR = {Fulton, William},
     TITLE = {Introduction to toric varieties},
    SERIES = {Annals of Mathematics Studies},
    VOLUME = {131},
 PUBLISHER = {Princeton University Press, Princeton, NJ},
      YEAR = {1993},
     PAGES = {xii+157},
      ISBN = {0-691-00049-2},
   MRCLASS = {14M25 (14-02 14J30)},
  MRNUMBER = {1234037},
MRREVIEWER = {T.\ Oda},
       DOI = {10.1515/9781400882526},
       URL = {https://doi.org/10.1515/9781400882526},
}

@book {CoxToric,
    AUTHOR = {Cox, David A. and Little, John B. and Schenck, Henry K.},
     TITLE = {Toric varieties},
    SERIES = {Graduate Studies in Mathematics},
    VOLUME = {124},
 PUBLISHER = {American Mathematical Society, Providence, RI},
      YEAR = {2011},
      ISBN = {978-0-8218-4819-7},
   MRCLASS = {14M25 (05A15 05E45 52B12)},
  MRNUMBER = {2810322},
MRREVIEWER = {Ivan\ Arzhantsev},
       DOI = {10.1090/gsm/124},
       URL = {https://doi.org/10.1090/gsm/124},
}

@article {EL93,
    AUTHOR = {Ein, Lawrence and Lazarsfeld, Robert},
     TITLE = {Syzygies and {K}oszul cohomology of smooth projective
              varieties of arbitrary dimension},
  JOURNAL = {Inventiones Mathematicae},
    VOLUME = {111},
      YEAR = {1993},
    NUMBER = {1},
     PAGES = {51--67},
      ISSN = {0020-9910,1432-1297},
   MRCLASS = {13D02 (14F17 14J60)},
  MRNUMBER = {1193597},
MRREVIEWER = {Gary\ P.\ Kennedy},
       DOI = {10.1007/BF01231279},
       URL = {https://doi.org/10.1007/BF01231279},
}

@book{Lazarsfeld04,
title={Positivity in algebraic geometry {I}: Classical setting: line bundles and linear series},
author={Lazarsfeld, Robert},
volume={48},
year={2004},
publisher={Springer}
}

@article {OgataNakagawa2002,
    AUTHOR = {Ogata, Shoetsu and Nakagawa, Katsuyoshi},
     TITLE = {On generators of ideals defining projective toric varieties},
  JOURNAL = {Manuscripta Mathematica},
    VOLUME = {108},
      YEAR = {2002},
    NUMBER = {1},
     PAGES = {33--42},
      ISSN = {0025-2611,1432-1785},
   MRCLASS = {14M25 (52B20)},
  MRNUMBER = {1912946},
MRREVIEWER = {Henry\ K.\ Schenck},
       DOI = {10.1007/s002290200252},
       URL = {https://doi.org/10.1007/s002290200252},
}

@article {Gubeladze2012,
    AUTHOR = {Gubeladze, Joseph},
     TITLE = {Convex normality of rational polytopes with long edges},
  JOURNAL = {Advances in Mathematics},
    VOLUME = {230},
      YEAR = {2012},
    NUMBER = {1},
     PAGES = {372--389},
      ISSN = {0001-8708,1090-2082},
   MRCLASS = {52B20 (14M25)},
  MRNUMBER = {2900547},
MRREVIEWER = {Milena\ S.\ Hering},
       DOI = {10.1016/j.aim.2011.12.003},
       URL = {https://doi.org/10.1016/j.aim.2011.12.003},
}

@article{Oda08,
      title={Problems on {M}inkowski sums of convex lattice polytopes}, 
      author={Tadao Oda},
      journal={arXiv preprint arXiv:0812.1418},
      year={2008}
}

@article{SWZ24,
      title={On covering simplices by dilations in dimensions 3 and 4}, 
      author={Song,Lei and Wen, Huanqi and Zhu, Zhixian},
      journal={arXiv preprint arXiv:2404.02495},
      year={2024}
}

@article{SW26,
      title={On {P}roperty ${N}_p$ of line bundles on smooth projective toric varieties}, 
      author={Song,Lei and Wen, Huanqi},
      journal={arXiv preprint arXiv:2606.30160},
      year={2026}
}

\bigskip
\noindent\small{\textsc{School of Mathematics, Sun Yat-sen University\\
W. 135 Xingang Rd., Guangzhou, Guangdong 510275, P.R.~China}\\
\emph{E-mail address}:  \texttt{songlei3@mail.sysu.edu.cn}

\bigskip
\noindent\small{\textsc{School of Mathematics, Sun Yat-sen University\\
W. 135 Xingang Rd., Guangzhou, Guangdong 510275, P.R.~China}\\
\emph{E-mail address}:  \texttt{wenhq7@mail2.sysu.edu.cn}

\bigskip
\noindent\small{\textsc{Academy for Multidisciplinary Studies, Capital Normal University\\
No. 105 West 3rd Ring Road, Beijing 100048, P.R.~China}\\
\emph{E-mail address}:  \texttt{zhixian@cnu.edu.cn}

\end{document}